\documentclass[10pt]{amsart}

\usepackage{amsmath,amsfonts,amssymb,amsthm}

\newtheorem{theorem}{Theorem}[section]
\newtheorem{lemma}[theorem]{Lemma}
\newtheorem{proposition}[theorem]{Proposition}

\newtheorem{claim}[theorem]{Claim}

\theoremstyle{definition}

\theoremstyle{remark}

\numberwithin{equation}{section}

\newcommand{\wo}{\mathbb{E}}
\newcommand{\wwo}[2]{\mathbb{E}\left[\left.{#1}\right|{#2}\right]}

\newcommand{\rzecz}{\mathbb{R}}
\newcommand{\zesp}{\mathbb{C}}

\DeclareMathOperator{\Var}{Var}                                                 
\newcommand{\wVar}[2]{\Var\left[\left.{#1}\right|{#2}\right]}                   
\DeclareMathOperator{\supp}{supp}                                               
\DeclareMathOperator{\MojeIm}{Im}                                               

\newcommand{\Fsu}{\mathcal{F}_{s,u}}

\newcommand{\Fs}{\mathcal{F}_{s}}
\newcommand{\Ft}{\mathcal{F}_{t}}

\newcommand{\yxts}{y;x,t,s}

\newcommand{\xx}{\mbox{\fontfamily{phv}\selectfont x}}
\newcommand{\xxx}{\emph{\xx}}
\newcommand{\yy}{\mbox{\fontfamily{phv}\selectfont y}}
\newcommand{\yyy}{\emph{\yy}}
\newcommand{\XX}{\mbox{\fontfamily{phv}\selectfont X}}
\newcommand{\XXX}{\emph{\XX}}
\newcommand{\DD}{\mbox{\fontfamily{phv}\selectfont D}}
\newcommand{\DDD}{\emph{\DD}}
\newcommand{\ZZ}{\mbox{\fontfamily{phv}\selectfont Z}}
\newcommand{\ZZZ}{\emph{\ZZ}}
\newcommand{\I}{\mbox{\fontfamily{phv}\selectfont I}}

\begin{document}

\title[Free Harness]{Free Quadratic Harness}

\author[W.~Bryc]{W\l odzimierz Bryc}
\address{Department of Mathematical Sciences\\
University of Cincinnati\\
Cincinnati, OH 45221-0025\\
USA} \email{Wlodzimierz.Bryc@UC.edu}

\author[W.~Matysiak]{Wojciech Matysiak}
\address{Wydzia{\l} Matematyki i Nauk Informacyjnych\\
Politechnika Warszawska\\
Pl. Politechniki 1\\
00-661 Warszawa, Poland}
\email{matysiak@mini.pw.edu.pl}

\author[J.~Weso\l owski]{Jacek Weso\l owski}
\address{Wydzia{\l} Matematyki i Nauk Informacyjnych\\
Politechnika Warszawska\\
Pl. Politechniki 1\\
00-661 Warszawa, Poland} \email{wesolo@mini.pw.edu.pl}
\thanks{}



\date{}

\dedicatory{}

\begin{abstract}
Free quadratic harness is a Markov process from the class of quadratic harnesses, i.e. processes with linear regressions and quadratic conditional variances. The process has recently been constructed for a restricted range of parameters in \cite{brycwesolo4} using Askey--Wilson polynomials. Here we provide a self-contained construction of the free quadratic harness for all values of parameters. 
\end{abstract}

\maketitle

\section{Introduction}\label{s:intro}

Quadratic harnesses were introduced in \cite{BMW1} as the square-integrable stochastic processes on $[0,\infty)$ such
that for all $t,s\ge0$
\begin{equation}\label{covariances}
\wo[X_t]=0,\ \wo[X_t X_s]=\min(t,s),
\end{equation}
conditional expectations $\wwo{X_t}{\Fsu}$ are linear functions of $X_s$ and $X_u$, and second conditional moments
$\wwo{X_t^2}{\Fsu}$ are quadratic functions of $X_s$ and $X_u$
\begin{equation}\label{quadvar}
\wwo{X_t^2}{\Fsu}=Q_{t,s,u}\left(X_s,X_u\right),
\end{equation}
where
\begin{equation}\label{quadform}
Q_{t,s,u}(x,y)=A_{t,s,u} x^2+B_{t,s,u} xy+C_{t,s,u} y^2+D_{t,s,u} x+E_{t,s,u} y+F_{t,s,u},
\end{equation}
and $A_{t,s,u},\ldots,F_{t,s,u}$ are some deterministic functions of $0<s<t<u$. Here, $\Fsu$ is the two-sided $\sigma$-algebra
generated by $\left\{X_r:\ r\in(0,s]\cup[u,\infty)\right\}$. We will also use the one sided $\sigma$-algebras $\Ft$
generated by $\left\{X_r:\ r\le t\right\}$.

It follows that for all $s<t<u$
\begin{equation}\label{linreg}
\wwo{X_t}{\Fsu}=a_{t,s,u}X_s+b_{t,s,u} X_u,
\end{equation}
with $a_{t,s,u}=(u-t)/(u-s)$ and $b_{t,s,u}=(t-s)/(u-s)$, and, under certain technical assumptions, there exist five
parameters $\eta,\theta\in\rzecz$, $\sigma,\tau\ge0$ and $\gamma\in[-1,1+2\sqrt{\sigma\tau}]$ such that
\begin{equation}\label{AtsuFtsu}
\begin{split}
A_{t,s.u} &= \frac{(u-t)\left[u(1+\sigma t)+\tau-\gamma t\right]}{(u-s)\left[u(1+\sigma s)+\tau-\gamma s\right]},\\
B_{t,s,u} &= \frac{(u-t)(t-s)(1+\gamma)}{(u-s)\left[u(1+\sigma s)+\tau-\gamma s\right]},\\
C_{t,s,u} &= \frac{(t-s)\left[t(1+\sigma s)+\tau-\gamma s\right]}{(u-s)\left[u(1+\sigma s)+\tau-\gamma s\right]},\\
D_{t,s,u} &= \frac{(u-t)(t-s)(u\eta-\theta)}{(u-s)\left[u(1+\sigma s)+\tau-\gamma s\right]},\\
E_{t,s,u} &= \frac{(u-t)(t-s)(\theta-s\eta)}{(u-s)\left[u(1+\sigma s)+\tau-\gamma s\right]},\\
F_{t,s,u} &= \frac{(u-t)(t-s)}{u(1+\sigma s)+\tau-\gamma s},
\end{split}
\end{equation}
and
\begin{multline}\label{condVar}
\wVar{X_t}{\Fsu}\\
= \frac{(u-t)(t-s)}{u(1+\sigma s)+\tau-\gamma s}\left(1+ \sigma
\frac{(uX_s-sX_u)^2}{(u-s)^2}+\eta \frac{uX_s-sX_u}{u-s} \right. \\
 \left.
+\tau\frac{(X_u-X_s)^2}{(u-s)^2}+\theta\frac{X_u-X_s}{u-s} +
(1-\gamma)\frac{(X_u-X_s)(sX_u-uX_s)}{(u-s)^2} \right),
\end{multline}
see \cite[Theorem 2.2]{BMW1}.

Quadratic harnesses may have orthogonal martingale polynomials (see \cite{BMW1} for the assumptions and \cite{jamiolkowska} for some exceptions), some explicit
examples of which have been worked out in Section 4 of \cite{BMW1}, for some of them the corresponding quadratic
harnesses were constructed in a series of papers \cite{BMW2}, \cite{brycwesolo1}, \cite{brycwesolo2}. A recent
development in proving the existence of quadratic harnesses is \cite{brycwesolo4}, where the machinery of
Askey--Wilson polynomials have been used to construct the processes for a wide range of parameters
$\eta,\theta,\sigma,\tau$ and $\gamma$.

However, in some cases, the theory developed in \cite{brycwesolo4}, brings some unnecessary limitations for the values
of parameters assuring the existence of the given quadratic harness. One of them is the case of free quadratic harness
defined in section 4.1 of \cite{BMW1}. Free harnesses have parameter
\begin{equation}
\gamma=-\sigma\tau;
\end{equation}
their orthogonal martingale polynomials were identified in
\cite[Proposition 4.3]{BMW1}, for $\sigma,\tau\ge0$,
$\sigma\tau<1$ and $1+\alpha\beta>0$, where
\begin{equation}\label{alphabeta}
\alpha=\frac{\eta+\theta\sigma}{1-\sigma\tau},\ \beta=\frac{\eta\tau+\theta}{1-\sigma\tau}.
\end{equation}
On the other hand, Corollary 5.3 in \cite{brycwesolo4}, which discusses the range of parameters
that guarantee the existence of the free harness, requires additional assumption
\begin{equation}\label{unnecessary}
2+\eta\theta+2\sigma\tau\ge0
\end{equation}
to be compliant with the theory developed in \cite{brycwesolo4}; nevertheless, the univariate Askey-Wilson
distributions are still well defined when $2+\eta\theta+2\sigma\tau<0$.

The goal of this paper is to show that free quadratic harness exists without assumption \eqref{unnecessary} (this is stated in the main result of the paper -- Theorem \ref{main}). The technique we use is similar to the one used in previously mentioned work (\cite{BMW2}, \cite{brycwesolo1},
\cite{brycwesolo2}), although various 
details differ.  We rely on explicit  three step
recurrences for the orthogonal martingale polynomials and on explicit connection coefficients between related families
of orthogonal polynomials. 
We also use an operator representation to prove the quadratic harness property of the constructed process.
The paper is self-contained and does not use any results from \cite{brycwesolo4}.

Our main result is

\begin{theorem}\label{main}
For $\sigma,\tau\ge0$, $\sigma\tau<1$, $\gamma=-\sigma\tau$, and
$1+\alpha\beta>0$, there exists a 
Markov process $(X_t)_{t\in[0,\infty)}$ such that \eqref{covariances},
\eqref{linreg} 
and \eqref{condVar} hold. The process $(X_t)_{t\in[0,\infty)}$ is unique among the processes with infinitely-supported one-dimensional distributions that have moments of all orders and satisfy \eqref{covariances}, \eqref{linreg} and \eqref{condVar} with the same parameters $\eta$, $\theta$, $\sigma$, $\tau$ and $\gamma=-\sigma\tau$.
\end{theorem}

The proof of this theorem is given in Section \ref{s:constr} after all auxiliary technical results are established.


It is worth mentioning that for some values of parameters, the univariate laws of free quadratic harnesses  are the first component of a two-state free convolution  semigroup  (for the free bi-Poisson process case see \cite{brycwesolo2}; for an extension to the case of $\sigma=0$ see \cite[Proposition 5]{AnshelevichMlotkowski2}).

\section{Orthogonal martingale polynomials and $q$-commutation equation}\label{s:orthopolyfreeharn}

This section presents a heuristic principle that can be used to find the recurrence for the polynomials orthogonal with respect to the conditional law $\mathcal{L}(X_t|X_s)$. (For the proof of Theorem \ref{main} such a derivation is not needed, as the actual recurrence used in the proof can be accepted as a guess.)

In \cite{BMW1} we defined the orthogonal martingale polynomials associated with the process $(X_t)_t$ as martingale polynomials (i.e. such that
\begin{equation}\label{op}
\wwo{p_n(X_t;t)}{\Fs}=p_n(X_s;s),
\end{equation}
holds, whenever $0\le s\le t$) that are orthogonal with respect to the one-dimensional distributions of the process. Theorem 2.3 from \cite{BMW1} states that their Jacobi matrix is linear in $t$,
\begin{equation*}
\mathbf{C}_t=
\begin{pmatrix}
\gamma_0 t+\delta_0 & \varepsilon_1 t+\phi_1 & 0  & 0 &  \ldots\\
\sigma\alpha_1t+\beta_1 & \gamma_1t+\delta_1 & \varepsilon_2t+\phi_2  & 0 & \ldots \\
0 & \sigma\alpha_2t+\beta_2 & \gamma_2t+\delta_2  & \varepsilon_3t+\phi_3 & \ldots \\
0 & 0 & \sigma\alpha_3t+\beta_3  & \gamma_3t+\delta_3  & \ldots \\
\vdots & \vdots & \vdots & \vdots  & \ddots
\end{pmatrix},
\end{equation*}
and for $n\ge1$ the coefficients $\alpha_n$, $\beta_n$, $\gamma_n$, $\delta_n$, $\phi_n$, $\varepsilon_n$ satisfy
\begin{equation}\label{eq1}
\sigma^2\tau\alpha_n\alpha_{n+1}
+\sigma\alpha_n\beta_{n+1}\gamma+\sigma\beta_n\beta_{n+1}=\sigma\alpha_{n+1}\beta_n,
\end{equation}
\begin{multline}\label{eq2}
\beta_{n+1}\gamma_{n+1}+\sigma\alpha_{n+1}\delta_n
=\sigma\alpha_{n+1}(\gamma_n+\gamma_{n+1})\tau+(\sigma\alpha_{n+1}\delta_{n+1}+\beta_{n+1}\gamma_n)\gamma+ \\
\beta_{n+1}(\delta_n+\delta_{n+1})\sigma +\sigma\alpha_{n+1}\theta+\beta_{n+1}\eta,
\end{multline}
\begin{multline}\label{eq3}
\beta_{n+1}\varepsilon_{n+1}+\gamma_n\delta_n+\sigma\alpha_{n}\varphi_n
=(\sigma\alpha_{n+1}\varepsilon_{n+1}+\gamma_n^2+\sigma\alpha_{n}\varepsilon_n)\tau
+\\(\sigma\alpha_{n+1}\varphi_{n+1}+\gamma_n\delta_n+\beta_{n}\varepsilon_n)\gamma
+(\beta_{n+1}\varphi_{n+1}+\delta_n^2+\beta_{n}\varphi_n)\sigma+\gamma_n\theta+\delta_n\eta+1,
\end{multline}
\begin{equation}\label{eq4}
\gamma_{n-1}\varphi_n+\delta_n\varepsilon_n
=(\gamma_{n-1}+\gamma_n)\varepsilon_n\tau+(\gamma_n\varphi_n+\delta_{n-1}\varepsilon_n)\gamma
+(\delta_{n-1}+\delta_n)\varphi_n\sigma+\varepsilon_n\theta+\varphi_n\eta,
\end{equation}
\begin{equation}\label{eq5}
\varepsilon_{n}\varphi_{n+1}=\varepsilon_{n}\varepsilon_{n+1}\tau+\varepsilon_{n+1}\varphi_{n}\gamma+\varphi_{n}\varphi_{n+1}\sigma,
\end{equation}
with the initial values given by
\begin{equation}\label{inipn}
\alpha_1=0,\ \beta_1=1,\ \gamma_0=\delta_0=0,\ \varepsilon_1=1,\ \phi_1=0.
\end{equation}

Now, for a fixed $r>0$, conditionally on $X_r$, the process 
$(Y_t^{(r)})_{t>0}$
\[
Y_t^{(r)}=\sqrt{\frac{1+\sigma r}{1+\eta X_r+\sigma X_r^2}}\left(X_{r+t}-X_r\right)
\]
still is a quadratic harness. Therefore when one considers the conditional distribution $\mathcal{L}(X_t|X_r)$ and the corresponding orthogonal polynomials $Q_n(y;x,t,r)$, \cite[Theorem 2.3]{BMW1} implies that their Jacobi matrix is again linear in $t$ and  that its entries again satisfy relations \eqref{eq1}-\eqref{eq5}. The only difference is that the initial values for the sequences should be modified as follows:
\begin{equation}\label{iniqn}
\alpha_1=0,\ \beta_1=1,\ \gamma_0=0,\ \delta_0=x,\ \varepsilon_1=\frac{1+\eta x+\sigma x^2}{1+\sigma r},\ \phi_1=\frac{-r\left(1+\eta x+\sigma x^2\right)}{1+\sigma r},
\end{equation}
as we choose the first polynomials $Q_n$ as
\begin{equation*}
Q_{-1}(y;x,t,s)\equiv0,\ Q_0(y;x,t,s)\equiv1,\ Q_1(y;x,t,s)=y-x,
\end{equation*}
and
\begin{multline*}
Q_2(y;x,t,s)=y^2\frac{1+\sigma s}{1+\sigma t}-\\
y\Biggl\{x\frac{(1+\sigma s)(1+\gamma)}{1+\sigma s+\sigma(\gamma s-\tau)} + \frac{(1+\sigma
s)\left[\theta+\eta\tau+t(\eta+\sigma\theta)-(1+\gamma)s\eta\right]} {(1+\sigma t)\left[1+\sigma
s+\sigma(\gamma s-\tau)\right]}\Biggr\} + \\
x^2\frac{\gamma+\sigma\tau}{1+\sigma s+\sigma(\gamma s-\tau)} + x
\frac{\theta+\eta\tau+s(\eta+\sigma\theta)-(1+\gamma)s\eta}{1+\sigma s+\sigma(\gamma s-\tau)}-\frac{t-s}{1+\sigma t}.
\end{multline*}

One can use equations \eqref{eq1}-\eqref{eq5} with initial values \eqref{iniqn} to derive the recurrences for the polynomials $Q_n$ from several previously studied cases from \cite{BMW1}, \cite{BMW2} and \cite{brycwesolo2}.

Here we are interested in the free harness case $\gamma+\sigma\tau=0$. After a calculation we get

 \begin{proposition}\label{QnFreeHarn}
 Suppose $\sigma,\tau\ge0$, $\sigma\tau<1$, $1+\alpha\beta> 0$,
 (see \eqref{alphabeta}) and
$\gamma=-\sigma\tau$.
Then recurrences \eqref{eq1}-\eqref{eq5} with initial condition \eqref{iniqn} have a  solution which defines the following three-step recurrence for polynomials $(Q_n)$ in variable $y$; here $s>0$ and $x\in\mathbb{R}$ are parameters.
\begin{multline}
\label{Q1} y Q_1(\yxts)=(1+\sigma t) Q_2(\yxts)+\left(\frac{\alpha+\sigma x}{1+\sigma
s}t+\frac{\beta-s(\eta+\sigma x)}{1+\sigma s}\right)Q_1(\yxts) \\
+\frac{(t-s)(1+\eta x+\sigma x^2)}{1+\sigma s}Q_0(\yxts)
\end{multline}
\begin{multline} \label{Q2} y Q_2(\yxts)=(1+\sigma t)
Q_3(\yxts)+ \frac{(\alpha+\sigma \beta)t+\beta+\alpha\tau}{1-\sigma
\tau}Q_2(\yxts)\\
+\frac{(t+\tau)(1+\alpha\beta)}{(1+\sigma s)(1-\sigma \tau)}Q_1(\yxts)
\end{multline}
\begin{multline} \label{Qn} y Q_n(\yxts)=(1+\sigma t)
Q_{n+1}(\yxts)+ \frac{(\alpha+\sigma \beta)t+\beta+\alpha\tau}{1-\sigma
\tau}Q_n(\yxts)\\
+\frac{(t+\tau)(1+\alpha\beta)}{(1-\sigma \tau)^2}Q_{n-1}(\yxts),\quad n\geq 3,
\end{multline}
with $Q_0\equiv1$ and $Q_1(y;x,t,s)=y-x$.
\end{proposition}



Since $p_n(y;t)=Q_n(y;0,t,0)$, it is not surprising that the above recurrence coincides in this case with \cite[Proposition 4.3]{BMW1} which we cite here for ease of reference in the proofs below.
\begin{proposition}[{\cite[Proposition 4.3]{BMW1}}]\label{Prop free biM}
Suppose $(X_t)_t$ is a quadratic harness with parameters such that
$\sigma,\tau\ge0$, $\sigma\tau<1$, $1+\alpha\beta> 0$, and
$\gamma=-\sigma\tau$. If for $t>0$ the random variable $X_t$ has all
moments and infinite support,
then
it has orthogonal martingale polynomials $(p_n)_n$ given by the three step recurrences 
\begin{align*}
yp_1(y;t)&= (1+\sigma t)p_2(y;t)+(\alpha t+\beta) p_1(y;t)+tp_0(y;t),\\
y p_2(y;t)&=(1+\sigma t)p_3(y;t)+\frac{(\alpha+\sigma\beta)t+\beta+\alpha\tau}{1-\sigma\tau} p_2(y;t)
+\frac{(t+\tau)(1+\alpha\beta)}{1-\sigma\tau}p_1(y;t),\\
y p_n(y;t)&=(1+\sigma t)p_{n+1}(y;t)+ \frac{(\alpha+\sigma\beta)t+\beta+\alpha\tau}
  {1 - \sigma \tau  }p_n(y;t)
  \\& \hspace{4cm}+
 \frac{(t+\tau)(1+\alpha\beta)}{(1-\sigma\tau)^2}p_{n-1}(y;t), \quad n\geq 3, \label{Free n}
\end{align*}
with $p_0\equiv1$ and $p_1(y;t)=y$.
\end{proposition}

Thus the recurrences for the polynomials $(Q_n)_n$ and $(p_n)_n$ turn out to be some finite perturbations of the constant coefficient recurrence. Therefore $(Q_n)_n$ and $(p_n)_n$ turn out to be Bernstein--Sz\"{e}go polynomials (see \cite[\S 2.6]{szego}), which are orthogonal with respect to the probability measure with the absolutely continuous part of the form $\sqrt{ax^2+bx+c}/\rho(x)$, where $\rho$ is a polynomial.

\section{One dimensional distributions}\label{s:onedimdistr}

Let $\pi_t$ denote the 
 orthogonality measure 
 of the polynomials $(p_n(y;t))_n$ with $t>0$, given by the three step recurrences in Proposition \ref{Prop free biM}. The existence of $\pi_t$ is assured by Favard's Theorem. In order to examine $\pi_t$ we will compute its Cauchy--Stieltjes transform $G_t$. Recall that the Cauchy--Stieltjes transform of a probability measure $\mu$ is an analytic mapping of the upper complex half-plane $\zesp_+$ into the lower half-plane $\zesp_-$ defined as
\[
G(z)=\int_\rzecz\frac{1}{z-x}\ \mu(\textrm{d}x).
\]
One of the properties of the Cauchy--Stieltjes transform that we shall use here is
\begin{equation}\label{cslimit}
\lim_{y\to\infty}iyG(iy)=1
\end{equation}
(see e.g. \cite{geronimohill}).
\begin{lemma}\label{lemma31}
If $\eta^2>4\sigma>0$, $\theta^2>4\tau>0$ and $\alpha+\sigma\beta>0$ then
\[
\pi_t\Biggl(\left(\frac{-\eta-\sqrt{\eta^2-4\sigma}}{2\sigma},\frac{-\eta+\sqrt{\eta^2-4\sigma}}{2\sigma}\right)\Biggr)=0.
\]
\end{lemma}

\begin{proof} It is known (see \cite{askeyismailmemoirs}, (2.10) and Theorem 2.4) that with orthogonal polynomials
one can associate a continued fraction (built from the
coefficients of the three term recurrence), which, if it converges,
is equal to the Cauchy--Stieltjes transform of the orthogonality
measure of the polynomials. In the case of the recurrence from
Proposition \ref{Prop free biM}, one gets
\begin{equation}\label{gduze}
G_t(z)=\cfrac{1}{z-\cfrac{t}{z-(\alpha t+\beta)-\frac{(1+\alpha\beta)(t+\tau)(1+\sigma t)}{1-\sigma\tau}g_t(z)}},
\end{equation}
where $g_t$ has the continued fraction expansion
\[
g_t(z)=\cfrac{\frac{1}{1+\sigma t}}{\frac{1}{1+\sigma t}z-\frac{(\alpha+\sigma\beta)t+\beta+\alpha\tau}{(1+\sigma t)(1-\sigma\tau)}-
\cfrac{\frac{(1+\alpha\beta)(t+\tau)}{(1+\sigma t)(1-\sigma\tau)^2}}{\frac{1}{1+\sigma t}z-\frac{(\alpha+\sigma\beta)t+\beta+\alpha\tau}{(1+\sigma t)(1-\sigma\tau)}-\cfrac{\frac{(1+\alpha\beta)(t+\tau)}{(1+\sigma t)(1-\sigma\tau)^2}}{\frac{1}{1+\sigma t}z-\ldots}}},
\]
so $g_t$ itself is the Cauchy--Stieltjes transform of a measure and it satisfies the quadratic equation
\[
g_t(z)=\cfrac{1}{z-\cfrac{(\alpha+\sigma\beta)t+\beta+\alpha\tau}{1-\sigma\tau}-
\cfrac{(1+\alpha\beta)(t+\tau)(1+\sigma t)}{(1-\sigma\tau)^2}\cdot g_t(z)}.
\]
(To justify convergence of the continued fraction expansion of $g_t$, one can use e.g. Theorem 2.1 \cite{askeyismailmemoirs}.) Hence
\begin{multline}\label{gmale}
g_t(z)=\frac{1-\sigma\tau}{2(1+\alpha\beta)(t+\tau)(1+\sigma t)}
\Biggl((1-\sigma\tau)z-(\alpha+\sigma\beta)t-\beta-\alpha\tau\\
\pm\sqrt{\Bigl((1-\sigma\tau)z-(\alpha+\sigma\beta)t-\beta-\alpha\tau\Bigr)^2-4(1+\alpha\beta)(t+\tau)(1+\sigma t)}\Biggr).
\end{multline}
By the square root in \eqref{gmale} we understand
\[
\sqrt{az-b-2\sqrt{c}}\cdot
\sqrt{az-b+2\sqrt{c}}=a\sqrt{z-\frac{b+2\sqrt{c}}{a}}\cdot\sqrt{z-\frac{b-2\sqrt{c}}{a}}
\]
where $a=1-\sigma\tau>0$, $b=(\alpha+\sigma\beta)t+\beta+\alpha\tau\in\rzecz$, $c=(1+\alpha\beta)(t+\tau)(1+\sigma t)>0$, the mappings
$z\mapsto\sqrt{z-(b\pm2\sqrt{c})/a}$ are analytic on $\zesp\setminus\left\{(b\pm2\sqrt{c})/a-t\ :\ t>0\right\}$, and take positive values for $\rzecz\ni z>(b+2\sqrt{c})/a$ (so they are the principal   branches of the square root, composed with linear transformations $z\mapsto z-(b\pm2\sqrt{c})/a$). Hence, by \eqref{cslimit}, one has to choose the "$-$" sign in the "$\pm$" in \eqref{gmale} .

Now, after inserting \eqref{gmale} into \eqref{gduze}, a calculation that uses \eqref{alphabeta} gives
\begin{multline}\label{duze-G} 
G_t(z)=\frac{\tau z+\theta t}{\tau z^2+\theta t z+t^2}+\frac{t\left[(1+\sigma\tau+2\sigma t)z+t\eta-\theta\right]}{2(\sigma z^2+\eta z+1)(\tau z^2+\theta t z+t^2)}\\
-\frac{t\sqrt{\left[(1-\sigma\tau)z-(\alpha+\sigma\beta)t-\beta-\alpha\tau\right]^2-4(1+\sigma t)(t+\tau)(1+\alpha\beta)}}{2(\sigma z^2+\eta z+1)(\tau z^2+\theta t z+t^2)}.
\end{multline}
Stieltjes--Perron inversion formula (see e.g. \cite[Theorem 2.5 and Section 2.3]{askeyismailmemoirs}) states that a finite Borel measure $\nu$ with the Cauchy--Stieltjes transform $G$ is absolutely continuous with
respect to Lebesgue measure on the set
\[
A=\{x:\lim_{\varepsilon\downarrow0} G(x+i\varepsilon)=\Phi(x),\textrm{ a finite number with }\MojeIm\Phi(x)\ne0\}.
\]
The atoms can only be located at simple poles of $G$ (see \cite{askeyismailmemoirs}). A very useful result (see \cite{reedsimon4}, Chapter XIII.6) states that if
\[
B=\{x:\lim_{\varepsilon\downarrow0} G(x+i\varepsilon)=\infty\},
\]
then $\nu(\rzecz\setminus(A\cup B))=0$ and $\nu$ restricted to $B$ is singular relative to Lebesgue measure. Therefore we see that the absolutely continuous part of $\pi_t$ is concentrated on the interval $[a_{-}(t),a_{+}(t)]$ with
\begin{equation*}
a_{\pm}(t)=\frac{(\alpha+\sigma\beta)t+\beta+\tau\alpha\pm2\sqrt{(1+\sigma t)(t+\tau)(1+\alpha\beta)}}{1-\sigma\tau},
\end{equation*}
the atoms can be located at (at most) four points being
zeros
of the polynomial $(\sigma z^2+\eta z+1)(\tau z^2+\theta t z+t^2)$:
\[
b_{\pm}=\frac{-\eta\pm\sqrt{\eta^2-4\sigma}}{2\sigma},\quad c_{\pm}(t)=-t\frac{\theta\pm\sqrt{\theta^2-4\tau}}{2\tau},
\]
and $\pi_t$ does not have continuous singular part.

Hence we will have established the lemma if we prove the following claims.
\begin{claim}\label{claimcont}
The continuous part of $\pi_t$ does not assign any probability to the interval $[b_{-},b_{+}]$.
\end{claim}
\begin{proof}
It suffices to check that $a_{-}(t)\ge b_{+}$ for $t>0$. It is easy to see that
\[
a_{-}^{\prime\prime}(t)=\frac{(1-\sigma\tau)\sqrt{1+\alpha\beta}}{2(1+\sigma t)^{3/2}(t+\tau)^{3/2}}>0.
\]
Note that $\alpha-\sigma\beta=\eta$ so
\[
t_\ast=\frac{(\alpha+\sigma\beta)(1-\sigma\tau)}{2\sigma\sqrt{(\alpha-\sigma\beta)^2-4\sigma}}-\frac{1+\sigma\tau}{2\sigma}
\]
is well defined.
The following bound shows that $t_\ast$ is in the domain of $a_-(t)$:
\[
(1+\sigma t_\ast)(t_\ast+\tau)=\frac{(1-\sigma\tau)^2(1+\alpha\beta)}{\eta^2-4\sigma}>0.
\]
Since $a'_-(t_\ast)=0$, and $a_-$ is convex, this is a global minimum of $a_-$.  It follows   that
 $a_{-}(t)\ge a_{-}(t_\ast)$ for $t>0$. Since $a_{-}(t_\ast)=b_{+}$, the proof is complete.
\end{proof}

\begin{claim}\label{claimdisc}
The discrete part of $\pi_t$ does not assign any probability to the interval $(b_{-},b_{+})$.
\end{claim}
\begin{proof}
If $\eta\theta<0$ then the points $b_\pm$ and $c_\pm(t)$ are separated by the interval $(a_{-}(t),a_{+}(t))$, so $c_\pm(t)\notin(b_{-},b_{+})$ for all $t>0$.

Suppose then that $\eta$ and $\theta$ are of the same sign. Since by \eqref{alphabeta}
\[
0<\alpha+\sigma\beta=\frac{\eta(1+\sigma\tau)+2\sigma\theta}{1-\sigma\tau},
\]
it follows that $\eta>0$ and $\theta>0$. The weights $p_\pm(t)$ of the points $c_\pm(t)$ are given by the residues of the Cauchy-Stieltjes transform $G_t$ at the points $c_\pm(t)$ (see \cite{askeyismailmemoirs}). A lengthy calculation reveals that
\begin{eqnarray*}
p_{-}(t)&=&\frac{2\tau\biggl(-t\left[2\eta\tau+(1+\sigma\tau)(\theta-\sqrt{\theta^2-4\tau})\right]+2\tau\sqrt{\theta^2-4\tau}\biggr)_{+}}
{\sigma\sqrt{\theta^2-4\tau}\left(\theta-\sqrt{\theta^2-4\tau}\right)^2
\left(t-\frac{\theta+\sqrt{\theta^2-4\tau}}{\eta-\sqrt{\eta^2-4\sigma}}\right)
\left(t-\frac{\theta+\sqrt{\theta^2-4\tau}}{\eta+\sqrt{\eta^2-4\sigma}}\right)},\\
p_{+}(t)&=&\frac{-2\tau\biggl(-t\left[2\eta\tau+(1+\sigma\tau)(\theta+\sqrt{\theta^2-4\tau})\right]-2\tau\sqrt{\theta^2-4\tau}\biggr)_{+}}{\sigma\sqrt{\theta^2-4\tau}\left(\theta+\sqrt{\theta^2-4\tau}\right)^2
\left(t-\frac{\theta-\sqrt{\theta^2-4\tau}}{\eta+\sqrt{\eta^2-4\sigma}}\right)
\left(t-\frac{\theta-\sqrt{\theta^2-4\tau}}{\eta-\sqrt{\eta^2-4\sigma}}\right)},
\end{eqnarray*}
where $(a)_{+}=(a+|a|)/2$. Clearly $p_{+}(t)=0$, and $p_{-}(t)>0$ only on the finite interval
\[
0\le t<\frac{2\tau\sqrt{\theta^2-4\tau}}{(\theta-\sqrt{\theta^2-4\tau})(1+\sigma\tau)+2\eta\tau};
\]
in particular,
\[
\frac{2\tau\sqrt{\theta^2-4\tau}}{(\theta-\sqrt{\theta^2-4\tau})(1+\sigma\tau)+2\eta\tau}<
\frac{\theta+\sqrt{\theta^2-4\tau}}{\eta+\sqrt{\eta^2-4\sigma}}.
\]
Since $c_{-}$ evaluated at the right hand side of the above inequality is equal to $b_{+}$, we get that the support of the discrete measure $p_{-}(t)\delta_{c_{-}(t)}$ stays above the level $b_{+}$ for all $t\ge0$.
\end{proof}
\noindent The proof of Lemma \ref{lemma31} is complete.
\end{proof}
In the next lemma we briefly describe 
the
extreme case of $\tau=0$.
\begin{lemma}\label{lemma44}
The assertion of Lemma \ref{lemma31} holds when $\eta^2>4\sigma>0$, $\tau=0$, $\theta^2>0$, and
$\alpha+\sigma\beta>0$.
\end{lemma}
\begin{proof}
The proof of Claim \ref{claimcont} carries over to the case $\tau=0$ without any changes.
Next we consider the atomic part of the measure.
Instead of two lines $c_\pm$, we have one
\[
c(t)=-\frac{t}{\theta}
\]
to take care of when examining the discrete part of $\pi_t$.
Since $\eta^2>0$ and $\theta^2>0$, 
as in the proof of Claim \ref{claimdisc} it suffices to consider the case of $\eta>0$ and $\theta>0$. The residue of $G_t$ at $c(t)$ is
\[
p(t)=\frac{\Bigl(-t\left[t(1+\eta\theta)-\theta^2\right]\Bigr)_{+}}{2\sigma\theta^2 t\left(t/\theta+b_{-}\right)\left(t/\theta+b_{+}\right)}.
\]
It follows that $p(t)>0$ for
\[
0\le t<\frac{\theta^2}{1+\eta\theta};
\]
in particular
\[
\frac{\theta^2}{1+\eta\theta}<-\theta b_{+},
\]
so the support of the discrete measure $p(t)\delta_{c(t)}$ stays above the level $b_{+}$ for all $t\ge0$.
\end{proof}

\section{Generating functions and connection coefficients}\label{s:genfunconncoeff}

Our next task is to establish an algebraic relation between the
polynomials $(p_n)_n$ and $(Q_n)_n$. In our setting, the relation
takes the same form as in \cite[Proposition 2.2]{brycwesolo2}; a
more complicated example occurs in \cite[Theorem 2.1]{BMW2}.

\begin{proposition}\label{prop41}
There exist polynomials $(b_k(x,s))_k$ and $(c_k(x,s))_k$ in variable $x$ such that $b_0(x,s)=1$ and
\begin{equation}\label{41}
Q_n(y;x,t,s)=c_n(x,s)+\sum_{k=1}^n b_{n-k}(x,s)p_k(y;t)
\end{equation}
and $(b_k)_k$ and $(c_k)_k$ do not depend on $t$ and $y$.
\end{proposition}

\begin{proof}
Let $\widehat{Q}$ denote the generating function of the polynomials $(Q_n)_n$, that is, let
\[
\widehat{Q}(z,y,x,t,s)=\sum_{n=0}^\infty z^n Q_n(y;x,t,s).\]
 From \eqref{41} with $y=0$ and $t=0$ we see that we must have $c_n(x,s)=Q_n(0;x,0,s)$.
 Therefore, to prove the proposition, we need only to verify that the right hand side of
\begin{equation}\label{42}
\widehat{b}(z,x,s)=\frac{\widehat{Q}(z,y,x,t,s)-\widehat{Q}(z,0,x,0,s)}{\widehat{Q}(z,y,0,t,0)-1}
\end{equation}
does not depend on variables $y$ and $t$. 
Then the series expansion $$\widehat{b}(z,x,s)=\sum_{n=0}^\infty z^n b_n(x,s)$$ defines the appropriate sequence $(b_n(x,s))_{n\geq 0}$.

To prove \eqref{42} we need an explicit formula for $\widehat{Q}$. Using the three step recurrence for $(Q_n)_n$, after a routine calculation one can verify that
\begin{multline}\label{43}
y\widehat{Q}=y Q_0+y z Q_1+z^2\frac{(t+\tau)(1+\alpha\beta)}{(1+\sigma s)(1-\sigma\tau)}Q_1+\frac{1+\sigma t}{z}\left(\widehat{Q}-Q_0-z Q_1-z^2 Q_2\right)+\\
+\frac{(\alpha+\sigma\beta)t+\beta+\alpha\tau}{1-\sigma\tau}\left(\widehat{Q}-Q_0-z Q_1\right)+
z\frac{(t+\tau)(1+\alpha\beta)}{(1-\sigma\tau)^2}\left(\widehat{Q}-Q_0-zQ_1\right)
\end{multline}
(to save space, we dropped the arguments $(z,y,x,t,s)$ in $\widehat{Q}$ and $(y;x,t,s)$ in $Q_n$). From this equation, after an elementary, but lengthy algebra, we first obtain a formula for $\widehat{Q}$ as a rational function of $z$, and then verify that \eqref{42} holds true with
\begin{multline*}
\widehat{b}(z,x,s)=\frac{1}{1+\sigma s}\cdot\\\Biggl(
\frac{z^2(1+\alpha\beta)s+z(1-\sigma\tau)\left[s(\alpha+\sigma\beta)-x(1-\sigma\tau)\right]+\sigma s(1-\sigma\tau)^2}{z^2\tau(1+\alpha\beta)+z(\beta+\alpha\tau)(1-\sigma\tau)+(1-\sigma\tau)^2}+1\Biggr).
\end{multline*}
Since $\widehat{b}(0,x,s)=1$, we get $b_0(x,s)=1$, as claimed.
\end{proof}

\section{Quadratic harness property}\label{s:quadharnprop}

In \cite{BMW1} we developed an operator approach, related to Lie algebra techniques, to the verification of the quadratic harness property. It uses a representation of the process under investigation through an operator $\XX_t=\xx+t\yy$, where $\xx$ and $\yy$ are some operators built from some compositions of the $q$-differentiation operator $\DD_q$ and the multiplication operator $\ZZ$.

Here we show how to exploit this technique to prove the quadratic harness property of the Markov process with martingale polynomials given in Proposition \ref{Prop free biM}. Let
\begin{equation}\label{analog434}
Q_{t,s,u}^\ast(\xx,\yy)=A_{t,s,u}\xx^2+B_{t,s,u}\yy\xx+C_{t,s,u}\yy^2+D_{t,s,u}\xx+E_{t,s,u}\yy+F_{t,s,u}
\end{equation}
be the quadratic form in the non-commuting variables $\xx$, $\yy$ (a dual of \eqref{quadform}). Define the generating function of the polynomials $(p_n)_n$ as $$\widehat{p}_t(z,y)=\sum_{n=0}^\infty z^n p_n(y;t).$$ In the free harness case, we are going to use the $0$-differentiation operator $\DD$, skipping the subscript, so
\[
\DD(g)(z)=\frac{g(z)-g(0)}{z}\ \textrm{ and }\ \ZZ(g)(z)=zg(z)
\]
(we treat them as the linear operators on formal series $g(z)$ in the variable $z$).

\begin{proposition}\label{analog410}
Let $\gamma=-\sigma\tau$ and
\begin{eqnarray*}
\xxx&=&\DDD+\beta\ZZZ\DDD+\frac{\tau(1+\alpha\beta)}{1-\sigma\tau}\ZZZ^2\DDD+\frac{\tau(\alpha+\sigma\beta)}{1-\sigma\tau}\ZZZ^2\DDD^2+\frac{\sigma\tau^2(1+\alpha\beta)}{(1-\sigma\tau)^2}{\ZZZ}^3{\DDD}^2,\\
\yyy&=&\ZZZ+\alpha\ZZZ\DDD+\sigma\ZZZ\DDD^2+\frac{\sigma(\beta+\alpha\tau)}{1-\sigma\tau}\ZZZ^2\DDD^2+\frac{\alpha\beta+\sigma\tau}{1-\sigma\tau}\ZZZ^2\DDD+
\frac{\sigma\tau(1+\alpha\beta)}{(1-\sigma\tau)^2}\ZZZ^3\DDD^2.
\end{eqnarray*}
The operator $\XXX_t=\xxx+t\yyy$ satisfies
\begin{equation}\label{analog435}
\XXX_t^2=Q^\ast_{t,s,u}(\XXX_s,\XXX_u)\quad\forall\ s<t<u,
\end{equation}
with the quadratic form given by \eqref{analog434} and \eqref{AtsuFtsu}. Moreover,
\begin{equation}\label{analog439}
y \widehat{p}_t(z,y)=\left(\XXX_t {{\widehat{p}}_t}\right)(z,y).
\end{equation}
\end{proposition}
\begin{proof}
A long but straightforward calculation shows that $\xx$ and $\yy$ satisfy the dual version of the $q$-commutation equation
\begin{equation}
[\xx,\yy]_\gamma=\sigma\xx^2+\tau\yy^2+\eta\xx+\theta\yy+\I.
\end{equation}
By \cite[Proposition 4.9]{BMW1}, \eqref{analog435} holds. The algebraic identity \eqref{analog439} follows from the three step recurrences for the polynomials $(p_n)_n$ given in Proposition \ref{Prop free biM}, by another routine calculation.
\end{proof}

\begin{proposition}\label{prop51}
If $(X_t)_t$ is a Markov process such that the random variables $X_t$ have moments of all orders and $(p_n)_n$ are orthogonal martingale polynomials of the process $(X_t)_t$, then $(X_t)_t$ is a quadratic harness with $\gamma=-\sigma\tau$.
\end{proposition}

\begin{proof}
Condition \eqref{covariances} holds true. Indeed,
\[
\wo X_t=\wo\left[p_1(X_t;t)p_0(X_t;t)\right]=0.
\]
For $s<t$, by the martingale property \eqref{op} and the first recurrence in 
Proposition \ref{Prop free biM} we get
\begin{multline*}
\wo[X_s X_t]=\wo\left[X_s\wwo{p_1(X_t;t)}{\Fs}\right]=
\wo\left[X_s p_1(X_s;s)\right]=\\
\wo\left[ (1+\sigma s)p_2(X_s;s)+(\alpha s+\beta)p_1(X_s;s)+s p_0(X_s;s)\right]=s.
\end{multline*}

An efficient way to verify \eqref{linreg} and \eqref{quadvar} is to use \eqref{analog439} to represent the process through the operator $\XX_t$ from Proposition \ref{analog410} as
\[
X_t {{\widehat{p}}_t}(z,X_t)=\XX_t\left({{\widehat{p}}_t}(z,X_t)\right).
\]
This, together with the martingale polynomial property, which for the generating function $\widehat{p}_t$ implies 
\[
\wwo{{\widehat{p}_t}(\xi,X_t)}{X_s}={\widehat{p}_s}(\xi,X_s),
\]
gives for $s\le t\le u$
\begin{multline*}
\wo\left({{\widehat{p}}_s}(\zeta,X_s)X_t{{\widehat{p}}_u}(\xi,X_u)\right)=
\wo\left({\widehat{p}_s}(\zeta,X_s)X_t{\widehat{p}_t}(\xi,X_t)\right)=\\
\XX_t\wo\left({\widehat{p}_s}(\zeta,X_s){\widehat{p}_s}(\xi,X_s)\right)=
\XX_t G_s(\zeta,\xi),
\end{multline*}
where
\[
G_s(\zeta,\xi)=\wo\left({\widehat{p}_s}(\zeta,X_s){\widehat{p}_s}(\xi,X_s)\right)=\sum_{n=0}^\infty(\zeta\xi)^n\wo\left(p_n(X_s;s)\right)^2,
\]
and $\XX_t$ acts on $G_s(\zeta,\xi)$ as on a series in variable $\xi$. Thus we arrive at the equivalence of
\begin{multline}\label{analog441}
\wo\left({{\widehat{p}}_s}(\zeta,X_s)X_t{{\widehat{p}}_u}(\xi,X_u)\right)=
a_{t,s,u}\wo\left({{\widehat{p}}_s}(\zeta,X_s)X_s
{{\widehat{p}}_u}(\xi,X_u)\right)
\\+
b_{t,s,u}\wo\left({{\widehat{p}}_s}(\zeta,X_s)X_u{{\widehat{p}}_u}(\xi,X_u)\right)
\end{multline}
and
\[
\XX_t G_s(\zeta,\xi)=a_{t,s,u}\XX_s G_s(\zeta,\xi)+b_{t,s,u}\XX_u G_s(\zeta,\xi).
\]
The latter (and so \eqref{analog441}) follows from the 
operator identity $\XX_t=a_{t,s,u}\XX_s+b_{t,s,u}\XX_u$, 
which is a trivial consequence of the representation $\XXX_t=\xxx+t\yyy$, $t>0$.

Now \eqref{analog441} means that
 \begin{multline*}
  \wo\left({{{p}}_n}(X_s;s)X_t{{{p}}_m}(X_u;u)\right)\\=
a_{t,s,u}\wo\left({{{p}}_n}(X_s;s)X_s{{{p}}_m}(X_u;u)\right)+
b_{t,s,u}\wo\left({{{p}}_n}(X_s;s)X_u{{{p}}_m}(X_u;u)\right)
\end{multline*}
for all $m,n\ge 0$. Since the random variables $X_t$ are bounded, polynomials are dense in $L^2(X_s,X_u)$ (see \cite[Theorem 3.1.18]{DunklXu}). Thus by the fact that $(X_t)_t$ is Markov,  \eqref{linreg} follows from \eqref{analog441}.

The proof of \eqref{quadvar} is similar, since \eqref{quadvar} is equivalent to
\begin{multline}\label{analog443}
\wo\left({\widehat{p}_s}(\zeta,X_s)X_t^2{\widehat{p}_u}(\xi,X_u)\right)=
A_{t,s,u}\wo\left({\widehat{p}_s}(\zeta,X_s)X_s^2{\widehat{p}_s}(\xi,X_s)\right)\\+
B_{t,s,u}\wo\left({\widehat{p}_s}(\zeta,X_s)X_s X_u{\widehat{p}_u}(\xi,X_u)\right)+
C_{t,s,u}\wo\left({\widehat{p}_s}(\zeta,X_s)X_u^2{\widehat{p}_u}(\xi,X_u)\right)\\+
D_{t,s,u}\wo\left({\widehat{p}_s}(\zeta,X_s)X_s{\widehat{p}_s}(\xi,X_s)\right)+
E_{t,s,u}\wo\left({\widehat{p}_s}(\zeta,X_s)X_u{\widehat{p}_u}(\xi,X_u)\right)\\+
F_{t,s,u}\wo\left({\widehat{p}_s}(\zeta,X_s){\widehat{p}_s}(\xi,X_s)\right).
\end{multline}
Observe that if $s\le u$ then 
\begin{multline}\label{abcd}
\wo\left({\widehat{p}_s(\zeta,X_s)X_sX_u{\widehat{p}}_u(\xi,X_u)}\right)=\wo\left({\widehat{p}_s(\zeta,X_s)X_s\XX_u{\widehat{p}}_u(\xi,X_u)}\right)=\\
\XX_u\wo\left({\widehat{p}_s(\zeta,X_s)X_s{\widehat{p}}_s(\xi,X_s)}\right)=
\XX_u\XX_s\wo\left({\widehat{p}_s(\zeta,X_s){\widehat{p}}_s(\xi,X_s)}\right)=\XX_u\XX_s G_s(\zeta,\xi).
\end{multline}
Similarly, $\XX_v^2 G_s(\zeta,\xi)=\wo\left({\widehat{p}_s}(\zeta,X_s)X_v^2{\widehat{p}_u}(\xi,X_u)\right)$ for $v\in[s,u]$. This and
\eqref{abcd} show that \eqref{analog443} follows from the operator identity \eqref{analog435} applied to $G_s(\zeta,\xi)$ treated as a formal power series in variable $\xi$, proving \eqref{quadvar}.

\end{proof}

\section{Construction and uniqueness}\label{s:constr}
Now we are in a position to prove the main result of this paper.
\begin{proof}[Proof of Theorem \ref{main}]
We first note that  since the time
inversion $(t X_{1/t})_t$ of the quadratic harness $(X_t)_t$ is
still a quadratic harness with parameters $\eta$, $\sigma$
replaced by $\theta$, $\tau$ (see Remark 2.1 in \cite{BMW1}),
it does not matter whether we construct $(X_t)_{t>0}$, or its time inversion $(t
X_{1/t})_{t>0}$. Secondly, we note that
  if $\eta^2>4\sigma>0$ and $\theta^2>4\tau>0$ then it is impossible to have simultaneously $\alpha+\sigma\beta=0$ and $\beta+\alpha\tau=0$  (recall \eqref{alphabeta}). So passing to time inversion if necessary, we may assume that $\alpha+\sigma\beta\ne 0$, and passing to $(-X_t)$ if necessary, we may assume, $\alpha+\sigma\beta> 0$.
  Similarly, observe that if $\sigma>0$, $\tau=0$ and $\eta^2>4\sigma$ then $\alpha+\sigma\beta\ne0$. Indeed, if $\tau=0$ and $\alpha+\sigma\beta=0$ then $\eta=-2\sigma\theta$; $\eta^2>4\sigma$ implies $\sigma\theta^2>1$ while $1+\alpha\beta>0$ implies $\sigma\theta^2<1$ - a contradiction. Hence we may assume $\alpha+\sigma\beta> 0$ as before.
Therefore, without loss of generality, we will consider the following list of constraints for the parameters for which we want to construct the quadratic harness:
\begin{itemize}
\item  Case 1: $\sigma,\tau>0$ and  $\eta^2\le4\sigma$,
\item  Case 2: $\sigma,\tau>0$ and  $\eta^2>4\sigma$, $\theta^2>4\tau$, and $\alpha+\sigma\beta> 0$,
\item Case 3: $\sigma>0$, $\tau=0$, $\eta^2\le4\sigma$,
\item Case 4: $\sigma>0$, $\tau=0$, $\eta^2>4\sigma$,  and $\alpha+\sigma\beta> 0$,
\item Case 5: $\sigma>0$, $\tau=\theta=0$,
\item Case 6: $\sigma=\tau=0$.
\end{itemize}

 We omit Cases 5 and 6, as the full construction of the quadratic harness with  $\sigma=\tau=0$ appeared in \cite{brycwesolo2} and the case $\tau=\theta=0$ is 
 the time-inversion of \cite[Theorem 4.3]{brycwesolo1}.
In the remaining cases, we will use
 polynomials $(p_n)_n=(p_n(y;t))_n$ from 
Proposition \ref{Prop free biM} to determine measures $\pi_t$ which will be the univariate laws of $(X_t)$. The orthogonality  measures $\mathbf{P}_{s,t}(x,\textrm{d}y)$  of  the   polynomials $(Q_n)_n=(Q_n(y;x,t,s))_n$ from Proposition \ref{QnFreeHarn} will be the transition probabilities of $(X_t)$ , i.e. the conditional laws $\mathcal{L}(X_t|X_s=x)$. We will verify that these probabilities satisfy the Chapman--Kol\-mo\-go\-rov equation, so that $(X_t)_t$ is indeed a well defined  Markov process.

It is clear that the coefficients at $Q_1$ in \eqref{Q2} and $Q_{n-1}$ at \eqref{Qn} are nonnegative.
So by Favard's theorem, in order  to define  probability measure  $\mathbf{P}_{s,t}(x,\textrm{d}y)$, we only need to check that the coefficient at $Q_0$ in \eqref{Q1} is nonnegative for $x$ from the support of the measure $\pi_s$.

 The coefficient at $Q_0$
is obviously nonnegative in Cases 1 and 3, as $\eta^2\le4\sigma $.
By Lemma \ref{lemma31}, the coefficient at $Q_0$ in \eqref{Q1} is
nonnegative in Case 2.
By Lemma \ref{lemma44} the coefficient at $Q_0$ in \eqref{Q1} is nonnegative in Case 4.


Thus, in 
each case, the polynomials $(Q_n(y;x,t,s))_n$ determine the probability measures $\mathbf{P}_{s,t}(x,\textrm{d}y)$ for all $x\in\supp\pi_s$. Observe that both families of measures $(\pi_t)_t$ and $(\mathbf{P}_{s,t}(x,\textrm{d}y))_{s,t,x}$ are 
compactly supported and uniquely determined, as the coefficients of the three step recurrences \eqref{Q1}-\eqref{Qn} are bounded 
in $n$. 

We now verify that the probability measures $(\mathbf{P}_{s,t}(x,\textrm{d}y))$ are the transition probabilities of a Markov process. To do so, notice that \eqref{41} for $n\ge1$ implies
\begin{equation}\label{62}
Q_n(y;x,t,s)=\sum_{k=1}^n b_{n-k}(x,s)\left[p_k(y;t)-p_k(x;s)\right]\quad\forall\ x,y\in\rzecz.
\end{equation}
(Observe that $Q_n(x;x,s,s)=0$  as a consequence of \eqref{Q1}-\eqref{Qn}.)
Since $b_0\equiv1$ and $p_0\equiv1$, a recursive use of \eqref{62} yields
\begin{equation}\label{63}
\int_\rzecz p_n(y;t)\mathbf{P}_{s,t}(x,\textrm{d}y)=p_n(x;s)\quad\forall\ x\in\supp\pi_s.
\end{equation}
Let
\[
U=\begin{cases}
\rzecz\setminus\left(\frac{-\eta-\sqrt{\eta^2-4\sigma}}{2\sigma},\frac{-\eta+\sqrt{\eta^2-4\sigma}}{2\sigma}\right),&
\textrm{when }\sigma>0,\\
\rzecz\setminus(-\eta^{-1},\infty),& \textrm{when }
\sigma=0\textrm{ and }\eta<0,\\
\rzecz\setminus(-\infty,\eta^{-1}),& \textrm{when }
\sigma=0\textrm{ and }\eta>0,\\
\rzecz,& \textrm{when } \sigma=0=\eta.
\end{cases}
\]
We proceed to show that for $0\le s<t<u$ and 
for a set of $x$ of
$\pi_s$-measure one
\begin{equation}\label{61}
\mathbf{P}_{s,u}(x,\cdot)=\int_U \mathbf{P}_{t,u}(y,\cdot)\mathbf{P}_{s,t}(x,\textrm{d}y).
\end{equation}
First, consider the special case $s=x=0$ of \eqref{61}, which we state equivalently as
\begin{equation}\label{65}
\pi_u(\cdot)=\int_U \mathbf{P}_{t,u}(y,\cdot)\pi_t(\textrm{d}y).
\end{equation}
Define $\nu(A)=\int_U \mathbf{P}_{t,u}(y,A)\pi_t(\textrm{d}y)$. To prove that $\nu(\textrm{d}z)=\pi_u(\textrm{d}z)$, we only need to show that the polynomials $Q_n(z;0,u,0)=p_n(z;u)$ are orthogonal with respect to $\nu(\textrm{d}z)$. Since the argument is analogous to the one developed in the general case below, we omit it.

From the fact that \eqref{61} holds for $s=x=0$, we deduce that
\begin{equation}\label{64}
\mathbf{P}_{s,t}(x,U)=1\quad\forall\ x\in\supp\pi_s.
\end{equation}
Indeed, observe first that since the coefficients in the three
step recurrences \eqref{Q1}-\eqref{Qn} depend continuously on $x$,
the same is true for the Cauchy--Stieltjes 
transforms of measures
$\mathbf{P}_{s,t}(x,\textrm{d}y)$, 
which take form \eqref{duze-G} with parameters that  depend on $s,x$, see \eqref{iniqn}.
So   $U\ni x\mapsto\mathbf{P}_{s,t}(x,U)$ is a continuous
function. Then, Lemma \ref{lemma31} and \eqref{65} imply that
\[
1=\pi_t(U)=\int_\rzecz \mathbf{P}_{s,t}(x,U)\pi_s(\textrm{d}x). 
\]
Therefore $\mathbf{P}_{s,t}(x,U)=1$ on a set of $x$ of
$\pi_s$-probability one. By continuity of $\mathbf{P}_{s,t}(x,U)$
in $x$, the conclusion follows for all $x\in\supp\pi_s$.

We now prove that \eqref{61} holds in general. Fix 
$s>0$ and
$x\in\supp\pi_s$, and let $\nu(\cdot)=\int_U \mathbf{P}_{t,u}(y,\cdot)\mathbf{P}_{s,t}(x,\textrm{d}y)$. We will show that $\nu(\textrm{d}z)=\mathbf{P}_{s,u}(x,\textrm{d}z)$ by checking that the polynomials $(Q_n(z;x,u,s))_n$ are orthogonal with respect to $\nu(\textrm{d}z)$. Since $\nu$ is a probability measure and $(Q_n)_n$ satisfy a three step recurrence with bounded coefficients, to verify that $\nu(\textrm{d}z)$ coincides with $\mathbf{P}_{s,t}(x,\textrm{d}z)$, it suffices to prove that $\nu(\textrm{d}z)$ integrates $Q_n(z;x,u,s)$ to zero when $n\ge1$. Using consecutively \eqref{62}, \eqref{63}, \eqref{64}, again \eqref{62}, and the fact that $\int_\rzecz Q_n(y;x,t,s)\mathbf{P}_{s,t}(x,\textrm{d}y)=0$ for $n\ge1$,  we get
\begin{multline*}
\int_\rzecz Q_n(z;x,u,s)\nu(\textrm{d}z)\\
=\int_U \sum_{k=1}^n b_{n-k}(x,s)\int_\rzecz\left[p_k(z;u)-p_k(x;s)\right] \mathbf{P}_{t,u}(y,\textrm{d}z)\mathbf{P}_{s,t}(x,\textrm{d}y)\\
=\sum_{k=1}^n b_{n-k}(x,s)\int_U\left[p_k(y;t)-p_k(x;s)\right]\mathbf{P}_{s,t}(x,\textrm{d}y)\\
=\sum_{k=1}^n b_{n-k}(x,s)\int_\rzecz\left[p_k(y;t)-p_k(x;s)\right]\mathbf{P}_{s,t}(x,\textrm{d}y)\\
=\int_\rzecz Q_n(y;x,t,s)\mathbf{P}_{s,t}(x,\textrm{d}y)=0.
\end{multline*}
Thus \eqref{61} holds and $\mathbf{P}_{s,t}(x,\textrm{d}y)$ are transition probabilities of a Markov process $(X_t)_t$ with state space $U$. Since $p_n(y;t)=Q_n(y;0,t,0)$ it follows from the construction that for fixed $t>0$ 
polynomials $(p_n(y;t))_n$ are orthogonal with respect to $\pi_t(\textrm{d}y)=\mathbf{P}_{0,t}(0,\textrm{d}y)$; their martingale polynomial property follows from \eqref{63}. Proposition \ref{prop51} implies that $(X_t)_t$ is a quadratic harness with parameters $\eta$, $\theta$, $\sigma$, $\tau$ and $\gamma$, with $\gamma=-\sigma\tau$.

Uniqueness of the process $(X_t)_t$ follows from the fact that orthogonal martingale polynomials $(p_n)_n$ determine uniquely the joint moments of the process. Recall that the  measures $\pi_t$ are compactly supported, so the joint moments determine the finite dimensional distributions of the process uniquely.
\end{proof}

\subsection*{Acknowledgement} The authors thank M. Bo{\.z}ejko and K. Szpojankowski for helpful discussions.
The research of W.~B. was partially supported by NSF  grant  DMS-0904720 and by Taft Research Center at the University of Cincinnati.

\bibliographystyle{amsplain}
\bibliography{c:/w/wm/res/var/toolbox/WM}

\end{document}